 \documentclass{amsart} 
 \usepackage{a4wide}
 \usepackage{times}
\usepackage{latexsym,amssymb,amsmath,hyperref}
\usepackage{xcolor}
\usepackage{bm}
\newcommand{\R}{{\mathbb R}}

\newcommand{\SF}{{\mathbb S}}

\newcommand{\dist}{{\rm dist}}

\newcommand{\C}{{\mathcal C}}

\newcommand{\K}{{\mathcal K}}

\newcommand{\B}{{\mathcal  B}}

\newcommand{\M}{{\mathcal M}}

\newcommand{\bA}{{\mathbf A}}
\newcommand{\bu}{{\mathbf u}}
\newcommand{\stilde}{{~}}
\def\loc{\mathrm{loc}}
\def\weak{\mathrm{weak}}

\def\Xint#1{\mathchoice
   {\XXint\displaystyle\textstyle{#1}}%
   {\XXint\textstyle\scriptstyle{#1}}%
   {\XXint\scriptstyle\scriptscriptstyle{#1}}%
   {\XXint\scriptscriptstyle\scriptscriptstyle{#1}}%
   \!\int}
\def\XXint#1#2#3{{\setbox0=\hbox{$#1{#2#3}{\int}$}
     \vcenter{\hbox{$#2#3$}}\kern-.5\wd0}}

\newtheorem{thm}{Theorem}[section]  %% Definition of Theorem
\newtheorem{lem}[thm]{Lemma}	       %% Definition of Lemma
\newtheorem{crlr}[thm]{Corollary}      %% Definition of Corollary
\newtheorem{prp}[thm]{Proposition}     %% Definition of Proposition
\newtheorem{defin}[thm]{Definition}    %% Definition of Definition
\newtheorem{rem}{Remark}	
\numberwithin{equation}{section}

\begin{document}

\title{Interior a priori estimates for 
higher--order elliptic systems   in Orlicz spaces}
\author{Amiran Gogatishvili, Pia Salerno, Lubomira Softova}

\address{A. Gogatishvili, 
Institute of Mathematics of the Czech Academy of Sciences 
25, \'Zitn\'a~, Prague~1, 
115~67, Praha, Czech Republic}
\email{gogatish@math.cas.cz, ORCID: 0000000334590355.}

\address{P. Salerno, Department of Mathematics, University of Salerno, 84084, Fisciano (SA), Italy}
\email{psalerno@unisa.it, ORCID: 0009000624391159.} 

\address{Corresponding author: L. Softova, Department of Mathematics, University of Salerno, 84084, Fisciano (SA), Italy} \email{lsoftova@unisa.it, ORCID: 0000000294989088.} 


\begin{abstract}
We investigate singular integral operators with variable Calder\'on--Zygmund kernels and their commutators with $VMO$ functions on Orlicz spaces. After revisiting the classical $L^p$ theory, we establish boundedness results in $L^\Phi$ under the standard $\Delta_2$ and $\nabla_2$ conditions on the Young function. The analysis combines decomposition techniques with weak-type estimates to derive the main operator bounds. As an application, these results provide a functional-analytic framework for establishing a priori estimates and proving the interior regularity of solutions to higher-order elliptic equations and systems with discontinuous coefficients.
\end{abstract}

\maketitle

\section{Introduction}\label{sec1}

Regularity theory for elliptic partial differential equations (PDEs) and systems with discontinuous coefficients has a long and rich history (see, e.g., \cite{BlS,BPSf,CZ1,CZ2,ChFF,DgN,LU,PSf1,PSf2,Sf1}). In the classical Lebesgue space setting, interior $W^{2b,p}$ estimates for higher-order linear elliptic operators with bounded measurable or $VMO$ coefficients were obtained by combining Korn's method of freezing coefficients with Calder\'on--Zygmund singular integral theory and commutator estimates.

This approach, initiated by the pioneering works of Chiarenza, Frasca, and Longo \cite{ChFrL1,ChFrL2}, has subsequently been extended in several directions. Bramanti and Cerutti \cite{BC} applied it to the 
Cauchy--Dirichlet problem for parabolic equations, Softova \cite{Sf2} obtained analogous results for higher-order parabolic equations, while Palagachev and Di Fazio \cite{DP1} established interior regularity for elliptic equations in non-divergence form. 
The case of quasilinear elliptic equations with $VMO$ coefficients was first investigated by Palagachev \cite{P} using a different method and was later further developed by Palagachev and Di Fazio \cite{DP2}. These works constitute the foundation of the modern regularity theory for elliptic and parabolic problems with discontinuous coefficients in the scale of $L^p$-Sobolev spaces.

More precisely, let $ \mathcal{L}(x,D)\mathbf{u}=\mathbf{f}$ 
be a system of differential operators that is uniformly elliptic in the sense of Agmon--Douglis--Nirenberg (cf.~\cite{DgN}) and whose coefficients belong to $VMO(\Omega)$. It is now well established that strong solutions satisfy the interior a priori estimate
$$
\|D^{2b}\mathbf{u}\|_{L^p(\Omega')}
\leq C\left(
\|\mathbf{f}\|_{L^p(\Omega'')} +
\|\mathbf{u}\|_{L^p(\Omega'')} \right),
\qquad 1<p<\infty,
$$
for every pair of subdomains $ 
\Omega'\Subset\Omega''\Subset\Omega.$ 
These estimates rely essentially on the boundedness of Calder\'on--Zygmund singular integral operators and their commutators in $L^p$ spaces, together with the vanishing mean oscillation of the coefficients at small scales. Furthermore, this regularity theory has been extended to the setting of Morrey and generalized Morrey--Sobolev spaces (see, for example  \cite{DPR,PRSf,PSf2}),  as well as subsequent contributions by Byun, Guliyev, and their collaborators.

Although the regularity theory in the $L^p$ setting is by now well established, its extension to function spaces with non-power growth remains far from complete. In many problems arising in the calculus of variations, non-linear partial differential equations, and models with non-standard growth, Orlicz and Orlicz--Sobolev spaces provide the natural analytical framework (see, for instance, \cite{CMa} and the references therein). However, the existing regularity results are largely restricted to scalar equations or second-order elliptic operators. To the best of our knowledge, a comprehensive interior regularity theory for higher-order elliptic systems in Orlicz spaces is still unavailable.

These spaces constitute a natural generalization of the classical Lebesgue scale, capable of accommodating non-power growth conditions and describing a broader class of function behaviours. Introduced by Orlicz in 1932, they have since evolved into a fundamental tool in modern analysis. Their theory is comprehensively presented in the monographs \cite{RR,RR2}, while the classical treatise \cite{KR} contains numerous applications to differential equations, integral operators, and function theory. Orlicz spaces play a central role in the study of problems with non-standard growth, including those arising in non-linear partial differential equations, fluid mechanics, and materials science.

The aim of this paper is to extend the classical interior regularity theory for higher-order elliptic systems with $VMO$ coefficients from the Lebesgue setting to the broader framework of Orlicz--Sobolev spaces associated with Young functions satisfying the standard $\Delta_2$ and $\nabla_2$ conditions.

More specifically, we consider linear elliptic systems of order $2b$, where $b\geq 1$ is an integer, whose matrix-valued coefficients belong to $VMO(\Omega)\cap L^\infty(\Omega)$ and satisfy the uniform ellipticity condition of Agmon--Douglis--Nirenberg. Assuming that the right-hand side $\mathbf{f}$ of \eqref{syst} belongs to the Orlicz space $L^\Phi(\Omega;\R^m)$, where $m\geq 1$ and the Young function $\Phi$ satisfies the standard $\Delta_2$ and $\nabla_2$ conditions, we establish interior a priori estimates in the Orlicz--Sobolev space $W^{2b}_\Phi(\Omega;\R^m)$ for strong solutions of the system.

Our main theorem establishes the following interior a priori estimate: if 
$ 
\mathbf{u}\in W^{2b}_\Phi(\Omega;\R^m)$ 
is a strong solution of \eqref{syst}, then for every pair of subdomains $ 
\Omega'\Subset\Omega''\Subset\Omega,$ 
there exists a constant $C>0$ such that
$$
\|\mathbf{u}\|_{W^{2b}_\Phi(\Omega')}
\leq C\left(
\|\mathbf{f}\|_{L^\Phi(\Omega'')} +
\|\mathbf{u}\|_{L^\Phi(\Omega'')}
\right).
$$

The constant $C$ depends only on the dimension, the order of the system, the ellipticity constant, the $L^\infty$-norms of the coefficients, and their $VMO$ modulus.

Methodologically, the proof follows the classical approach based on freezing the coefficients, representing solutions by means of the fundamental solution of the associated constant-coefficients operator, and estimating the resulting singular integral operators and commutators. The main challenge is to adapt this strategy to the framework of Orlicz spaces. To this end, we establish suitable interpolation inequalities under the standard $\Delta_2$ and $\nabla_2$ conditions on the Young function and combine them with a Campanato-type iteration argument to obtain uniform control of the lower-order derivatives.

The novelty of the present work does not lie in the underlying methodology, which is firmly rooted in the classical Calder\'on--Zygmund theory, but rather in the level of generality achieved. To the best of our knowledge, interior $W^{2b}_\Phi$ estimates for higher-order elliptic systems with VMO coefficients have not previously been established in the setting of Orlicz spaces associated with Young functions satisfying the standard $\Delta_2$ and $\nabla_2$ conditions. Our results demonstrate that the classical regularity theory for elliptic systems with $VMO$ coefficients is robust with respect to the underlying growth scale and extends naturally from the Lebesgue framework to a broad class of Orlicz spaces.

The paper is organized as follows. Sections~\ref{sec2} and~\ref{sec3} review the basic properties of Orlicz and Orlicz--Sobolev spaces and the boundedness theory of the Hardy--Littlewood maximal operator. Section~\ref{sec4} is devoted to the study of Calder\'on--Zygmund singular integral operators and their commutators in Orlicz spaces. Finally, Section~\ref{sec5} establishes the main interior regularity theorem for higher-order elliptic systems with $VMO$ coefficients.

We shall use the following notation throughout the paper.

\begin{itemize}

\item For $x=(x_1,\ldots,x_n)\in\R^n$, we denote by
$  |x|=\left(\sum_{i=1}^n x_i^2\right)^{1/2} $  the Euclidean norm.

\item For $x\in\R^n$ and $r>0$, $ 
\mathbb B_r(x)=\{y\in\R^n:|x-y|<r\}$ 
denotes the open Euclidean ball of radius $r$ centred at $x$. Its Lebesgue measure is $ 
| \B_r(x)|=\omega_n r^n,$ 
where $\omega_n=| \B_1(0)|$ is the volume of the unit ball.

\item
$ \SF^{n-1}= \{x\in\R^n:|x|=1\}$  denotes the unit sphere in $\R^n$, whose surface measure satisfies $ 
|\SF^{n-1}|=n\omega_n.$ 

\item For $1\leq p<\infty$, $L^p(\R^n)$ denotes the  Lebesgue space, while $
L^{p,\infty}(\R^n)\equiv L^p_{\weak}(\R^n) $
denotes the weak Lebesgue space equipped with the quasi-norm
$$
\|f\|_{L^p_{\mathrm{weak}}(\R^n)} =
\sup_{\lambda>0} \lambda \bigl|
\{x\in\R^n:|f(x)|>\lambda\} \bigr|^{1/p}.
$$

\item If $f\in L^1(D)$, where $D\subset\R^n$ is a measurable set of finite measure, we denote by
$$
f_D=
\frac1{|D|} \int_D f(y)\,dy= \Xint-_D f(y)\,dy
$$
the average of $f$ over $D$.

\item 
$\alpha=(\alpha_1,\dots,\alpha_n)$ is  a multi-index of order
$|\alpha|=\alpha_1+\cdots+\alpha_n$.

\item 
We use the notation
$D^\alpha=D_{x_1}^{\alpha_1}\cdots D_{x_n}^{\alpha_n}$, where
$D_{x_i}=\partial/\partial x_i$, and  write $D^{2b}$ for  the collection of all derivatives of order   $2b$.

\item Unless otherwise specified, vector-valued function spaces are understood component-wise, and write 
$ \|\cdot\|_{L^\Phi(\Omega)}$ 
instead of $ \|\cdot\|_{L^\Phi(\Omega;\R^m)}. $  For a vector function  $\bu$ we write 
 $$
\|D^{2b}\bu\|_{L^\Phi(\Omega)}:=\sum_{|\alpha|=2b}\|D^\alpha \bu\|_{L^\Phi(\Omega)}.
 $$

\end{itemize}

%----------------------------------------------------------
\section{Function Spaces and Auxiliary Results}\label{sec2}
Our objective is to investigate the local regularity of solutions to higher-order linear elliptic equations and systems. To this end, we introduce the function spaces and integral operators that constitute the main analytical tools in the derivation of the regularity estimates.

One of the fundamental operators in harmonic analysis is the \emph{Hardy--Littlewood maximal operator} $\M$, which is defined for every locally integrable function $f\in L^1_{\loc}(\R^n)$ by
\begin{equation}
\M f(x)=\sup_{r>0}\Xint-_{\B_r(x)}|f(z)|\,dz,
\qquad x\in\R^n,
\end{equation}
where the supremum is taken over all balls $\B_r(x)$ centred at $x$.

The Hardy--Littlewood maximal operator possesses several fundamental properties (see, for instance, \cite{GC,Gr,RR2,St}). In particular, it satisfies the pointwise estimate
$|f(x)|\leq \M f(x),$ 
for almost every $x\in\R^n,$
which follows directly from the Lebesgue Differentiation Theorem (cf. \cite{Gr,St}).
Moreover, the strong-type $(p,p)$ Hardy--Littlewood inequality asserts that, for every
$f\in L^p(\R^n)$ \mbox{with $1<p<\infty$,} one has
\begin{equation*}%\label{strong1}
\|\M f\|_{L^p(\R^n)}
\le
C_p\,\|f\|_{L^p(\R^n)}.
\end{equation*}

When $p=1$, the maximal operator satisfies the weak-type $(1,1)$ estimate
\begin{equation}\label{weak1}
\bigl|\{x\in\R^n:\M f(x)>t\}\bigr|
\le
\frac{C(n)}{t}
\int_{\R^n}|f(x)|\,dx, \qquad \forall\, t>0.
\end{equation}

To describe the regularity of the coefficients, we recall the spaces introduced by John and Nirenberg~\cite{JN} and later refined by Sarason~\cite{Sa}. 

A measurable locally integrable function $a$ is said to belong to the space $BMO$ (\emph{Bounded Mean Oscillation}) if the seminorm
\begin{equation}\label{BMO}
\|a\|_\ast = \sup_{\B_r(x)} \Xint-_{\B_r(x)}
|a(z)-a_{\B_r(x)}|\,dz
\end{equation}
is finite, where the supremum is taken over all balls $\B_r(x)\subset\R^n$. The quantity $\|\cdot\|_\ast$ is a seminorm on $BMO$ and induces a norm on the quotient space $ BMO/\R $ modulo constant functions, which is a Banach space.
Moreover, $ L^\infty(\R^n)\subset  BMO,$ and the estimate $ 
\|a\|_\ast\leq 2\|a\|_{L^\infty(\R^n)} $ 
holds for every $a\in L^\infty(\R^n)$.

A function  $a \in BMO$ belongs to the space $VMO$  ({\it Vanishing Mean Oscillation}) if 
\begin{equation}\label{VMO}
\eta_a(R) = \sup_{\B_r(y), \, r \leq R}  \Xint-_{\B_r(y)} |a(z) - a_{\B_r(y)}| \, dz
\end{equation}
tends to zero as $R$ tends to zero. The quantity $\eta_a(R)$ is called the $VMO$-modulus of $a.$

The following classical results summarize some fundamental properties of $BMO$ and $VMO$  functions.

\begin{lem}[John--Nirenberg lemma, \cite{JN}]\label{JN-type}
Let $a\in BMO$ and let $1\leq p<\infty$. Then, for every ball $\B_r(x)\subset\R^n$, one has
\begin{equation}\label{JN}
\left( \Xint-_{\B_r(x)} |a(z)-a_{\B_r(x)}|^p\,dz
\right)^{1/p} \leq 
C(p)\|a\|_\ast.
\end{equation}
\end{lem}

\begin{thm}[Sarason~\cite{Sa}]\label{SARASON}
Let $f\in BMO$. Then, the following assertions are equivalent:
\begin{itemize}
\item[(i)] $f\in VMO$;

\item[(ii)] $f$ belongs to the closure, with respect to the $BMO$ seminorm, of the space of bounded uniformly continuous functions;

\item[(iii)] \ 
$\displaystyle \lim_{y\to0} 
\|f(\cdot-y)-f(\cdot)\|_\ast=0.$
\end{itemize}
\end{thm}

%------------------------------------------------
\section{Orlicz spaces, Definitions and Fundamental Properties}\label{sec3}

We begin by recalling the fundamental properties of Young functions,  which play a central role in our analysis. These functions are \textit{convex} and provide a natural generalization of the power functions $\Phi(t)=t^p$ with $p>1$, which characterize the classical Lebesgue spaces (see, e.g., \cite{VK,KR}).

A function $\Phi:[0,\infty)\to[0,\infty)$ is called a
\emph{Young function} if it is convex, lower semicontinuous,
not identically zero, satisfies $\Phi(0)=0$, and
$$
\lim_{t\to\infty}\Phi(t)=\infty
$$

An equivalent characterization in terms of an integral representation can be found in \cite{PK,RR}.

A Young function $\Phi$ is called an \emph{$N$-function} if
\begin{equation}\label{Nfunction}
\lim_{t\to0^+}\frac{\Phi(t)}{t}=0,
\qquad
\lim_{t\to\infty}\frac{\Phi(t)}{t}=\infty.
\end{equation}

An $N$-function $\Phi$ satisfies
\begin{itemize}
\item[(i)] the \emph{$\Delta_2$-condition} if there exists a constant $\mu>1$ such that
\begin{equation}\label{def-delta2}
\Phi(2t)\le \mu\,\Phi(t),
\qquad \text{for all } t>0;
\end{equation}

\item[(ii)] the \emph{$\nabla_2$-condition} if there exists a constant $l>1$ such that
\begin{equation}\label{def-nabla2}
\Phi(t)\le \frac{1}{2l}\,\Phi(lt),
\qquad \text{for all } t>0.
\end{equation}
\end{itemize}

We write $\Phi\in\Delta_2\cap\nabla_2$ whenever both conditions are satisfied.
The $\Delta_2$- and $\nabla_2$-conditions provide a two-sided control on the growth of $\Phi$. As will become clear in the subsequent sections, both conditions are essential for establishing the regularity results proved in this paper. Accordingly, throughout the paper we assume that $\Phi\in\Delta_2\cap\nabla_2.$

For example, the Young function $\Phi_1(t)=t^p$, where $1<p<\infty$, satisfies the $\Delta_2$-condition for every $\mu\geq 2^p$ and the $\nabla_2$-condition for every $l\geq 2^{1/(p-1)}$. Other examples of Young functions are
$$
\Phi_2(t)=e^t-1, \qquad
\Phi_3(t)=t\ln(1+t), \qquad t\ge0.
$$

It is readily verified that $\Phi_2$ grows faster than any polynomial. Consequently,
$\Phi_2\in\nabla_2$ but $\Phi_2\notin\Delta_2$. On the other hand, $\Phi_3$ grows more slowly than any polynomial, but faster than linearly. Hence,
$\Phi_3\in\Delta_2$, whereas $\Phi_3\notin\nabla_2$.

In addition, the $\Delta_2$-condition implies (see, for instance, \cite{AG,VK,RR}) that, for every $\lambda>1$, there exists a constant $\mu_\lambda>1$ such that
\begin{equation}\label{lambda}
\Phi(\lambda t)\leq \mu_\lambda\,\Phi(t), 
\qquad \forall \ t>0.
\end{equation}

The $\nabla_2$-condition implies the quasi-convexity of $\Phi$ (see, for instance, \cite[Lemma~1.2.3]{VK} and \cite[Lemma~6.1.6]{AG}). Moreover, the following characterization, established in \cite[Lemma~1.1.1]{VK} will play a key role in our analysis.

A function $\phi:[0,\infty)\to\mathbb{R}$ is called \emph{quasi-convex} if there exist a convex function \mbox{$\omega:[0,\infty)\to\mathbb{R}$} and a constant $C_1>0$ such that
\begin{equation}\label{quasiconvexity}
\omega(t)\leq \phi(t)\leq C_1\,\omega(C_1t), \qquad \forall\, t\geq 0.
\end{equation}
Every convex function is quasi-convex, but the converse is not true in general.

Quasi-convexity admits several equivalent characterizations, two of which are particularly relevant in the theory of Orlicz spaces.
\begin{lem}[\text{\cite[Theorem~1.2.1]{VK}}]\label{pow2}
Let $\Phi$ be a Young function. Then, the following statements are equivalent:
\begin{enumerate}
\item $\Phi \in \nabla_2$;
\item there exists $\alpha \in (0,1)$ such that $\Phi^\alpha$ is quasi-convex.
\end{enumerate}
\end{lem}

\begin{lem}[\text{\cite[Lemma~1.1.1]{VK}}]\label{pow3}
Let $\Phi:[0,\infty)\to [0,\infty)$ be a function. 
\begin{enumerate}
\item 
If $\Phi$ is  quasi-convex, then 
 there exists a constant $\mathfrak{d}>1$ such that
\begin{equation}\label{power3}
\frac{\Phi(t_1)}{t_1}
\le
\frac{\mathfrak{d}\,\Phi(\mathfrak{d}t_2)}{t_2},
\qquad
0<t_1<t_2<\infty.
\end{equation}
\item 
If, in addition, $\Phi$ is a Young function, then the converse also holds.
\end{enumerate}
\end{lem}

In what follows, we present, for the reader's convenience, some of the main consequences of the $\Delta_2$- and $\nabla_2$-conditions for a Young function $\Phi$ separately (see \cite{AG,VK}). To this end, we first recall the definition of a quasi-increasing function.

A function $g:(0,\infty)\to\mathbb{R}$ is called \emph{quasi-increasing} if there exists a constant $C_2>0$ such that
\begin{equation}\label{quasincr}
g(t_1)\leq C_2\,g(t_2), \qquad \forall \ 0<t_1<t_2<\infty.
\end{equation}
If $\Phi$ is a Young function, then the function
$ t\mapsto \Phi(t)/t$ 
is increasing on $(0,\infty)$ and, therefore, quasi-increasing with constant $C_2=1$. Applying \eqref{quasincr} to this function with $t_1=t>0$ and $t_2=Ct$, where $C>1$, we obtain
 the quasi-convex growth estimate
\begin{equation}\label{CCC}
\Phi(t)\leq \frac{1}{C}\,\Phi(Ct), \qquad \forall\, t>0.
\end{equation}

\begin{lem}[\cite{KR}]\label{pow}
Let $\Phi\in\Delta_2$. Then, there exist constants $P>1$ and $\mathfrak{b}>1$ such that
\begin{equation}\label{power}
\frac{\Phi(t_2)}{t_2^P}
\le
\mathfrak{b}\,\frac{\Phi(t_1)}{t_1^P},
\qquad
0<t_1<t_2<\infty.
\end{equation}
Moreover, the constants $P$ and $\mathfrak{b}$ depend only on the $\Delta_2$-constant of $\Phi$.
\end{lem}

In other words, Lemma~\ref{pow} asserts that there exists $P>1$ such that the function
$  t\mapsto \Phi(t)/t^P$ 
is \emph{quasi-decreasing}. Consequently, for every $t_0>0$,
$$
\frac{\Phi(t)}{t^P}
\leq \mathfrak b\,\frac{\Phi(t_0)}{t_0^P}, \qquad  t_0<t<\infty,
$$
and hence the function $t\mapsto \Phi(t)/t^P$ is bounded on $[t_0,\infty)$.

Motivated by this result, we now establish an analogous property for Young functions satisfying the $\nabla_2$-condition.

\begin{lem}\label{cpow}
Let $\Phi\in\nabla_2$. Then, there exist constants $R>1$ and $\mathfrak{a}>1$ such that
\begin{equation}\label{cpower}
\frac{\Phi(t_1)}{t_1^R} \leq
\mathfrak{a}\,\frac{\Phi(\mathfrak{a}t_2)}{t_2^R},
\qquad
0<t_1<t_2<\infty.
\end{equation}
\end{lem}
\begin{proof}
By Lemma~\ref{pow2}, there exists $\alpha\in(0,1)$ such that $\Phi^\alpha$ is quasi-convex. Applying Lemma~\ref{pow3} to the function $\Phi^\alpha$, we conclude that there exists a constant $\mathfrak{d}>1$ such that
$$
\frac{\Phi^\alpha(t_1)}{t_1}
\leq \mathfrak{d}\,
\frac{\Phi^\alpha(\mathfrak{d}t_2)}{t_2},
\qquad 0<t_1<t_2<\infty.
$$
Raising both sides to the power $1/\alpha$, we obtain
$$
\frac{\Phi(t_1)}{t_1^{1/\alpha}}
\leq  \mathfrak{d}^{1/\alpha}\,
\frac{\Phi(\mathfrak{d}t_2)}{t_2^{1/\alpha}},
\qquad 0<t_1<t_2<\infty.
$$
Setting 
$ R=\frac{1}{\alpha}>1, $  $
\mathfrak{a}=\mathfrak{d}^{1/\alpha}>1, $
and using the monotonicity of $\Phi$, we obtain \eqref{cpower}.
\end{proof}

In particular, Lemma~\ref{cpow} asserts that there exists $R>1$ such that the function
$t\mapsto \Phi(t)/t^R$ 
is \emph{quasi-increasing}, in the sense of \eqref{quasincr}, up to a multiplicative dilation of the argument. Moreover, \eqref{cpower} ensures boundedness of the ratio   $\Phi(t)/t^R$ on every interval $(0,t_0].$ 

Moreover,  we have that  $1<R\leq P<\infty.$ Indeed, fixing  $t_1>0$, applying first  \eqref{cpower}, and then  \eqref{power} to the pair  $(t_1, \mathfrak{a}t_2)$,  after simplification,  we obtain
$$
t_2^{R-P}\leq \mathfrak{a}^{P+1}\mathfrak{b}\,t_1^{R-P}, \qquad \forall \ t_2>t_1>0.
$$
Because of the arbitrary of $t_2$ the above inequality holds if $R-P\leq  0.$

For our purposes (see Section~\ref{sec4}), we also require alternative integral characterizations of the $\nabla_2$- and $\Delta_2$-conditions in terms of Hardy-type inequalities.

\begin{lem}\label{lem-nabladelta}
Let $\Phi\in\Delta_2\cap\nabla_2$. Then, there exist exponents $  1<r<p<\infty$  such that
\begin{align}
\label{nabla2}
\int_0^t \frac{d\Phi(s)}{s^r}
&\leq C_r\,\frac{\Phi(t)}{t^r},\\[5pt]
\label{delta2}
\int_t^\infty \frac{d\Phi(s)}{s^p}
&\leq  C_p\,\frac{\Phi(t)}{t^p}. 
\end{align}
\end{lem}
\begin{proof}

The proof is based on classical Hardy-type arguments under the
$\Delta_2$- and $\nabla_2$-conditions.

Let $R$ and $P$ be respectively the exponents given by Lemmas~\ref{cpow} and~\ref{pow}, and fix $1<r<R$.

Since
$\Phi\in\nabla_2$, Lemma~\ref{cpow} implies that
$$
\lim_{s\to0_+}\frac{\Phi(s)}{s^r}=\lim_{s\to0_+}
\left(\frac{\Phi(s)}{s^R}\right)s^{R-r}=0.
$$
Integrating by parts and using \eqref{lambda} and \eqref{cpower}, we obtain
\allowdisplaybreaks
\begin{align*}
\int_0^t \frac{d\Phi(s)}{s^r}
&=\frac{\Phi(t)}{t^r} +
r\int_0^t\frac{\Phi(s)}{s^{r+1}}\,ds  \\
&=\frac{\Phi(t)}{t^r} + r\int_0^t
\frac{\Phi(s)}{s^R}\,
\frac{ds}{s^{\,r-R+1}} \\
&\leq \frac{\Phi(t)}{t^r} +
r\,\frac{\mathfrak a\,\Phi(\mathfrak a t)}{t^R}
\int_0^t s^{\,R-r-1} \, ds \\
&\leq \left(
1+\frac{r\,\mathfrak a\mu_{\mathfrak a}}{R-r}
\right)
\frac{\Phi(t)}{t^r}=:C_r\frac{\Phi(t)}{t^r}.
\end{align*}

Analogously, let $P < p < \infty$. Since $\Phi\in\Delta_2$, Lemma~\ref{pow} yields
$$
\lim_{s\to\infty}\frac{\Phi(s)}{s^p}=
\lim_{s\to\infty}
\left(\frac{\Phi(s)}{s^P}\right)\frac1{s^{p-P}}=0.
$$
Integrating by parts and using \eqref{power}, we obtain
\begin{align*}
\int_t^\infty\frac{d\Phi(s)}{s^p}
&= -\frac{\Phi(t)}{t^p}+ p\int_t^\infty
\frac{\Phi(s)}{s^{p+1}}\, ds  \\ &= -\frac{\Phi(t)}{t^p}+ p\int_t^\infty
\frac{\Phi(s)}{s^P}\, \frac{ds}{s^{\,p-P+1}}\\
&\leq 
\frac{\Phi(t)}{t^p} + p\,\mathfrak b\, \frac{\Phi(t)}{t^P} \int_t^\infty s^{P-p-1}\,ds \\
&= \left(1+\frac{p\,\mathfrak b}{p-P} \right)
\frac{\Phi(t)}{t^p} =:C_p\frac{\Phi(t)}{t^p}.
\end{align*}
\end{proof}

Related Hardy-type integral inequalities under the $\Delta_2$- and $\nabla_2$-conditions can be found in~\cite{Ci2}.

\begin{crlr}\label{cor1}
If \eqref{nabla2} holds for some $1<r<\infty$, then it also holds for every
$1<r_1<r$. Similarly, if \eqref{delta2} holds for some $1<p<\infty$, then it also holds for every $p_1>p$.
\end{crlr}
\begin{proof}
Let $1<r_1<r$. Since $s^{r-r_1}\leq t^{r-r_1}$ for every $0<s\leq t$, we obtain
$$
\int_0^t\frac{d\Phi(s)}{s^{r_1}}
= \int_0^t s^{\,r-r_1}\,
\frac{d\Phi(s)}{s^r}
\leq t^{\,r-r_1}
\int_0^t\frac{d\Phi(s)}{s^r}
\leq C_r\,t^{\,r-r_1}\frac{\Phi(t)}{t^r} =
C_r\,\frac{\Phi(t)}{t^{r_1}}.
$$

Similarly, let $p_1>p$. Since $s^{p-p_1}\leq t^{p-p_1}$ for every $s\geq t$, we have
$$
\int_t^\infty\frac{d\Phi(s)}{s^{p_1}}
\leq t^{\,p-p_1}
\int_t^\infty\frac{d\Phi(s)}{s^p}
\leq C_p\,\frac{\Phi(t)}{t^{p_1}}.
$$
\end{proof}

Finally, we recall the basic properties of Orlicz spaces and present several fundamental results that will be used throughout the paper  (see, for instance, \cite[Ch.~IX]{RR2}). As a natural generalization of the classical Lebesgue spaces, Orlicz spaces provide a flexible framework for describing non-standard growth and play a fundamental role in the analysis of integral operators and partial differential equations with non-standard growth conditions.

The \emph{complementary Young function}
$\Psi:[0,\infty)\to[0,\infty)$ associated with $\Phi$ is defined by
$$
\Psi(y):=\sup_{x\geq 0}\bigl\{xy-\Phi(x)\bigr\},
\qquad y\geq 0.
$$
The function $\Psi$ shares the main qualitative properties of $\Phi$: it is non-negative, convex,  increasing, and satisfies analogous growth conditions at both the origin and infinity. Moreover, for every pair of complementary Young functions $(\Phi,\Psi)$, Young's inequality holds:
\begin{equation}\label{disY}
xy\le\Phi(x)+\Psi(y), \qquad x,y\geq 0.
\end{equation}
The \emph{Orlicz class} $\bar{L}^\Phi(\mathbb{R}^n)$ consists of all measurable functions
$f:\mathbb{R}^n\to\mathbb{R}$ whose  \emph{Orlicz modular}
$$
\rho_\Phi(f):=\int_{\mathbb{R}^n}\Phi(|f(x)|)\,dx
$$
is finite.

\begin{defin}\label{Orliczspace}
The \emph{Orlicz space} $L^\Phi(\mathbb{R}^n)$ consists of all measurable functions
$f:\mathbb{R}^n\to\mathbb{R}$ such that
$\rho_\Phi(f/\lambda)<\infty $ for some $\lambda>0$.
It is equipped with the \emph{Luxemburg norm}
\begin{equation}\label{Lux}
\|f\|_{L^\Phi(\mathbb{R}^n)}
:=\inf\left\{\lambda>0:
\rho_\Phi\!\left(\frac{f}{\lambda}\right)\leq 1
\right\}.
\end{equation}
\end{defin}

\begin{prp}[Jensen's Inequality {\rm\cite[Proposition~5]{RR}}]\label{JensInOrl}
Let $\Omega\subset\mathbb{R}^n$ be a measurable set with $|\Omega|<\infty$, and let
$f\in L^\Phi(\Omega)$. Then
\begin{equation}\label{JensOrl}
\Phi\left(
\frac{1}{|\Omega|}
\int_\Omega |f(x)|\,dx
\right)
\le
\frac{1}{|\Omega|}
\int_\Omega \Phi(|f(x)|)\,dx.
\end{equation}
\end{prp}

The following Jensen-type inequality involving the Hardy--Littlewood maximal operator will play an important role in the sequel.

\begin{prp}\label{JensWMax}
Let $f\in L^\Phi(\mathbb{R}^n)$. Then, for almost every $x\in\mathbb{R}^n$,
$$
\Phi(\mathcal{M}f(x))
\leq \mathcal{M}\bigl(\Phi(|f|)\bigr)(x),
$$
where $\mathcal{M}$ denotes the Hardy--Littlewood maximal operator.
\end{prp}

\begin{proof}
For every ball $B_r(x)\subset\mathbb{R}^n$, Jensen's inequality
\eqref{JensOrl} yields
$$
\Phi\left( \frac1{|B_r(x)|}
\int_{B_r(x)}|f(z)|\,dz \right) \leq \frac1{|B_r(x)|}\int_{B_r(x)}\Phi(|f(z)|)\,dz.
$$
Taking the supremum over all $r>0$, and using the continuity and monotonicity of $\Phi$, we obtain
\begin{align*}
\Phi(\mathcal{M}f(x))
&=\sup_{r>0} \Phi\left( \frac1{|B_r(x)|}
\int_{B_r(x)}|f(z)|\,dz \right)\\
&\leq \sup_{r>0}\frac1{|B_r(x)|}
\int_{B_r(x)}\Phi(|f(z)|)\,dz\\
&= \mathcal{M}\bigl(\Phi(|f|)\bigr)(x),
\end{align*}
which completes the proof.
\end{proof}

We recall some classical boundedness properties of the Hardy--Littlewood maximal
operator on Orlicz spaces, which play a fundamental role in the study of
regularity and integral estimates under non-standard growth conditions (see, for instance,
\cite{RR,RR2}).

\begin{thm}[Weak-type inequality]\label{key1.1}
Let $\Phi$ be a Young function. Then, there exists a constant $C>0$ such that, for every $\alpha>0$ and every $f\in L^\Phi(\R^n)$, the Hardy--Littlewood
maximal operator $\M$ satisfies the weak-type $(\Phi,\Phi)$ estimate
\begin{equation}\label{weakorlicz}
\left|\left\{x\in\R^n:\M f(x)>\alpha\right\}\right|
\le
\frac{C}{\Phi(\alpha)}
\int_{\R^n}\Phi(|f(x)|)\,dx,
\end{equation}
where the constant $C$ is independent of both $f$ and $\alpha$.
\end{thm}

\begin{proof}
Since $\Phi$ is increasing, using Proposition~\ref{JensWMax} and the weak-type $(1,1)$ inequality \eqref{weak1} to the
non-negative function $\Phi(|f|)$, we obtain
\allowdisplaybreaks
\begin{align*}
	|\{x \in \R^n : \M  f(x) > \alpha \}| &=|\{x \in \R^n : \Phi(\M  f(x)) > \Phi(\alpha) \}| \\
	&\leq |\{x \in \R^n : \M  \, \Phi(|f(x)|) > \Phi(\alpha) \}| \\
	&\leq \frac{C}{\Phi(\alpha)} \int_{\R^n} \Phi(|f(x)|) \, dx.
\end{align*}
which completes the proof.
\end{proof}

\begin{thm}[Strong-type inequality]\label{key1.2}
Let $\Phi\in\nabla_2$. Then, there exists a constant $C>0$ such that, for every
$f\in L^\Phi(\R^n)$, the Hardy--Littlewood maximal operator $\M$ satisfies the
strong-type $(\Phi,\Phi)$ estimate
\begin{equation}\label{strongORLnorm}
\|\M f\|_{L^\Phi(\R^n)}
\le
C\,\|f\|_{L^\Phi(\R^n)}.
\end{equation}

Moreover, there exists a constant $C'>0$ such that
\begin{equation}\label{strongorlicz}
\int_{\R^n}\Phi(\M f(x))\,dx
\le
\int_{\R^n}\Phi(C'|f(x)|)\,dx.
\end{equation}
\end{thm}
\begin{proof}
The result is a standard consequence of the weak-type estimate established  in
Theorem~\ref{key1.1} together with the assumption $\Phi\in\nabla_2$; see, for instance,
\cite[Ch.~IX]{RR2}.
\end{proof}

%---------------------------------------------------------
\section{Singular Integrals in Orlicz Spaces}\label{sec4}

In this section, we study the boundedness properties of singular integral operators with variable Calder\'on--Zygmund kernels acting on Orlicz spaces.

\begin{defin}\label{dVCZ}
A function
$$
k(x;\xi):\R^n\times(\R^n\setminus\{0\})\to\R
$$
is called a \emph{variable Calder\'on--Zygmund (VCZ) kernel} if it satisfies the following conditions:

\begin{enumerate}

\item For every fixed $x\in\R^n$, the function $\xi\mapsto k(x;\xi)$ is a classical Calder\'on--Zygmund kernel, that is,
\begin{enumerate}
\item
$k(x;\cdot)\in C^\infty(\R^n\setminus\{0\});$
\item
$k(x;\cdot)$ is homogeneous of degree $-n$, namely,
$$
k(x;\mu\xi)=\mu^{-n}k(x;\xi),
\qquad \forall\,\mu>0,\ \xi\in\R^n\setminus\{0\};
$$
\item
$k(x;\cdot)$ satisfies the cancellation condition
$$
\int_{\mathbb S^{n-1}}k(x;\xi)\,d\sigma_\xi=0,
\qquad \int_{\mathbb S^{n-1}}|k(x;\xi)|\,d\sigma_\xi<\infty.
$$
\end{enumerate}

\item
For every multi-index $\beta$, there exists a constant $C(\beta)>0$ such that
$$
\sup_{x\in \R^n}\sup_{\xi\in\mathbb S^{n-1}}
|D^\beta_\xi k(x;\xi)| \leq C(\beta).
$$
\end{enumerate}
\end{defin}

\begin{lem}[Hörmander condition]
Let $\B=\B_r(x_0)$ and $2\B=\B_{2r}(x_0)$. Then, there exists a constant
$C=C(n)>0$ such that
$$
|k(x;x-y)-k(x;x_0-y)|
\leq  C\,\frac{|x-x_0|}{|x_0-y|^{n+1}}
$$
for every  $x\in\B$ and all $y\notin2\B$.
\end{lem}

The proof is standard and follows from the smoothness and homogeneity assumptions in 
Definition~\ref{dVCZ};  see, for instance,
\cite[Lemma~2.2]{ChFrL1}).

Given a variable Calder\'on--Zygmund kernel $k(x;\xi)$, we define the singular
integral operator
$$
\K f(x) :=\mathrm{p.v.}\int_{\R^n} k(x;x-y)\,f(y)\,dy,
$$
and the commutator
$$
\C[a,f](x) := \mathrm{p.v.}\int_{\R^n}
k(x;x-y)\,[a(y)-a(x)]\,f(y)\,dy,
$$
where $a\in L^\infty(\R^n)$ and $\mathrm{p.v.}$ denotes the Cauchy principal value.

The following theorem establishes the boundedness of the operators
$\K$ and $\C[a,\cdot]$ on $L^p(\R^n)$ and is proved in
\cite{ChFrL1}; see also \cite{ChFrL2,DP1,DPR}.

\begin{thm}\label{thm-3}
Let $1<p<\infty$, let $f\in L^p(\R^n)$, and let $a\in BMO.$  Then, there exists a constant $C=C(n,p)>0$ such that
$$
\|\K f\|_{L^p(\R^n)}
\leq  C\,\|f\|_{L^p(\R^n)},\qquad 
\|\C[a,f]\|_{L^p(\R^n)}
\leq  C\,\|a\|_\ast\,\|f\|_{L^p(\R^n)},
$$
where $\|a\|_\ast$ denotes the $BMO$ norm of $a$.

In addition, if $a\in VMO\cap L^\infty(\R^n)$ and $\eta_a $ denotes the $VMO$-modulus of $a$, then for any $\varepsilon>0$ there exist a positive constants $\rho_0=\rho_0(\varepsilon,\eta_a)$ and $C(n,p)$ such that 
\begin{equation}\label{local}
\|\C[a,f]\|_{L^p(\B_r)}\leq C\,\varepsilon \|f\|_{L^p(\B_r)}, \quad \forall \, f\in L^p(\B_r), 
\end{equation}
for every  ball $\B_r$ with radius $r\in (0,\rho_0)$.
\end{thm}

\begin{rem}
The assumption $a\in BMO$ is sharp for the boundedness of the commutator $\C[a,\cdot]$. 
\end{rem}

As a direct consequence of the strong $L^p$ boundedness, it follows that  the corresponding weak-type estimates; see, for instance, \cite{Gr}.

\begin{lem}\label{lem-1}
Let $1<p<\infty$ and let $f\in L^p(\R^n)$. Then, there exists a constant  $\kappa_p=\kappa(C,p)>0$, where $C$ is the constant appearing in Theorem~\ref{thm-3}, such that 
\begin{equation}\label{eq-Kf-weakp}
\|\K f\|_{L^{p,\infty}(\R^n)}
\leq  \kappa_p\,\|f\|_{L^p(\R^n)},
\end{equation}
equivalently, for every $t>0$,
\begin{equation}\label{eq-Kf-weakp2}
\left| \left\{ x\in\R^n: |\K f(x)|>t
\right\}\right|
\leq  \frac{\kappa_p^p}{t^p} \int_{\R^n}|f(x)|^p\,dx.
\end{equation}

Moreover, if $a\in BMO,$  then
\begin{equation}\label{eq-Caf-weakp}
\|\C[a,f]\|_{L^{p,\infty}(\R^n)}
\le
\kappa_p\,\|a\|_\ast\,\|f\|_{L^p(\R^n)},
\end{equation}
that is, for every $t>0$,
\begin{equation}\label{eq-Caf-weakp2}
\left|
\left\{
x\in\R^n:
|\C[a,f](x)|>t
\right\}
\right|
\le
\frac{\kappa_p^p\|a\|_\ast^p}{t^p}
\int_{\R^n}|f(x)|^p\,dx.
\end{equation}
\end{lem}
\begin{proof}
It is well known from the classical theory of Lorentz spaces (see, for instance,
\cite{Gr}) that
$$
L^p(\R^n) \hookrightarrow L^{p,\infty}(\R^n),
\qquad
\|g\|_{L^{p,\infty}(\R^n)} \leq \|g\|_{L^p(\R^n)},
$$
for every $g\in L^p(\R^n)$.
Applying this embedding with  $g=\K f$ and $g=\C[a,f]$, together with the strong
$L^p$ estimates of  Theorem~\ref{thm-3}, we obtain
\eqref{eq-Kf-weakp} and \eqref{eq-Caf-weakp}. The corresponding
distributional inequalities \eqref{eq-Kf-weakp2} and \eqref{eq-Caf-weakp2} follow directly from the definition of the weak $L^p$ norm.
\end{proof}

The $L^p$ boundedness of the singular integral operators extends to Orlicz spaces  under the assumptions $\Phi\in\Delta_2\cap\nabla_2$.

\begin{thm}\label{thm-4}
Let $\Phi\in\Delta_2\cap\nabla_2$, let $f\in L^\Phi(\R^n)$, and let $a\in BMO.$ 
Then, the operators $\K$ and $\C[a,\cdot]$ are bounded on $L^\Phi(\R^n)$. 
More precisely, there exists a constant $C>0$ such that
\begin{align}
\label{eq-Kfphi2}
\|\K f\|_{L^\Phi(\R^n)}
&\le
C\,\|f\|_{L^\Phi(\R^n)},\\[5pt]
\label{eq-CAFphi2}
\|\C[a,f]\|_{L^\Phi(\R^n)}
&\le
C\,\|a\|_\ast\,\|f\|_{L^\Phi(\R^n)}.
\end{align}
\end{thm}

\begin{proof}
Let $f\in L^\Phi(\R^n)$. We decompose $f$ as
$f=f_t+f^t,$ 
where
\begin{equation}\label{Dec}
f_t(x)=
\begin{cases}
f(x), & \text{if } |f(x)|\leq \dfrac{t}{C},\\[5pt]
0, & \text{if } |f(x)|>\dfrac{t}{C},
\end{cases}
\qquad
f^t(x)=
\begin{cases}
f(x), & \text{if } |f(x)|>\dfrac{t}{C},\\[5pt]
0, & \text{if } |f(x)|\leq\dfrac{t}{C}.
\end{cases}
\end{equation}
Here
$$
C:=\max\left\{
2^{1+\frac1r}\,\kappa_r\, C_r^{\frac1r}, \
2^{1+\frac1p}\,\kappa_p\, C_p^{\frac1p}
\right\},
$$
where $\kappa_r$, $C_r$, $\kappa_p$, and $C_p$ are the constants appearing in
Lemmas~\ref{lem-nabladelta} and~\ref{lem-1}, respectively. 
By the definition of $C$, we have
\begin{equation*}
\frac1{C^r}
\leq 
\frac1{2^{r+1}\kappa_r^rC_r},
\qquad
\frac1{C^p}
\leq 
\frac1{2^{p+1}\kappa_p^pC_p}.
\end{equation*}

Since 
$
|\K f(x)| \leq |\K f^t(x)| + |\K f_t(x)| $,  therefore 
$$ 
\{x \in \R^n \ : \ |\K f(x)|>t\} \subseteq \left\{x \in \R^n \ : \ |\K f^t(x)|> \tfrac{t}{2}\right\} \cup \left\{x \in \R^n \ : \ |\K f_t(x)|> \tfrac{t}{2}\right\}.
$$
We obtain the following estimate for the layer-cake representation
\begin{align*}
\int_{\R^n} \Phi(|\K  f(x)|) \, dx  &=\int_0^\infty |\{ x \in \R^n \ : \ |\K f(x)|>t\}| \, d\Phi(t) \\ 
&\leq \int_0^\infty \left|\left\{ x \in \R^n \ : \ |\K f^t(x)|>\tfrac{t}{2}\right\}\right| \, d\Phi(t) + \int_0^\infty \left|\left\{ x \in \R^n \ : \ |\K f_t(x)|>\tfrac{t}{2}\right\}\right| \, d\Phi(t)\\
& =: I_1 + I_2.
\end{align*}

Applying  the weak-type $(L^r,L^{r,\infty})$ estimate \eqref{eq-Kf-weakp2}  to $f^t$,   together with Fubini's theorem and Lemma~\ref{lem-nabladelta}, we estimate the first integral 
\begin{align*}
I_1 &
%\leq
%\int_0^\infty
%\left|\left\{x\in\R^n:\,|\K f^t(x)|>
%\tfrac{t}{2}\right\} \right|\,d\Phi(t) \\
%&
\leq  2^r\kappa_r^r
\int_0^\infty \frac1{t^r}
\int_{\{|f(x)|>\tfrac{t}{C}\}}
|f(x)|^r\,dx\, d\Phi(t) \\
&= 2^r\kappa_r^r \int_{\R^n}
|f(x)|^r \left( \int_0^{C|f(x)|}
\frac{d\Phi(t)}{t^r}\right)\, dx \\
&\leq  \frac12 \int_{\R^n} \Phi(C|f(x)|)\,dx.
\end{align*}
 
Similarly, applying  the weak-type $(L^p,L^{p,\infty})$ estimate 
\eqref{eq-Kf-weakp2} to $f_t$, together with Fubini's theorem and
Lemma~\ref{lem-nabladelta}, we obtain the estimate for the second integral 
\begin{align*}
I_2
&\leq 2^p\kappa_p^p
\int_0^\infty
\frac1{t^p} \int_{\{|f(x)|\leq \tfrac{t}{C}\}}
|f(x)|^p\,dx\,d\Phi(t) \\
&= 2^p\kappa_p^p
\int_{\R^n}
|f(x)|^p \left(
\int_{C|f(x)|}^\infty
\frac{d\Phi(t)}{t^p}
\right)\,  dx \\
&\leq  \frac12 \int_{\R^n} \Phi(C|f(x)|)\,dx.
\end{align*}

Combining the above estimates, yields 
$$
\int_{\R^n} \Phi(|\K f(x)|)\,dx
\leq  \int_{\R^n}\Phi(C|f(x)|)\,dx,
$$
which by the definition of the Luxemburg norm,   implies
$$
\|\K f\|_{L^\Phi(\R^n)}
\leq  C\,\|f\|_{L^\Phi(\R^n)}.
$$

The proof for the commutator $\C[a,\cdot]$ is analogous, but in the decomposition
\eqref{Dec}, we replace the term $\frac{t}{C}$ by
$\frac{t}{C\|a\|_\ast}$ and use the weak-type 
bounds for $\C[a,f]$. Proceeding as above, we write
\begin{align*}
\int_{\R^n}\Phi( |\C[a,f](x)|)\, dx &= \int_0^\infty \left|\left\{ x \in \R^n \ : \ |\C[a,f](x)|>t \right\}\right| \, d\Phi(t) \\ &\leq 
 \int_0^\infty \left|\left\{ x \in \R^n \ : \ |\C[a,f^t](x)|>\tfrac{t}{2}\right\}\right| \, d\Phi(t) \\
 &\quad + \int_0^\infty \left|\left\{ x \in \R^n \ : \ |\C[a,f_t](x)|>\tfrac{t}{2}\right\}\right| \, d\Phi(t) \\ &=:I_3+I_4.
 \end{align*}

Applying the weak-type estimate \eqref{eq-Caf-weakp2} for 
the commutator
$\C[a,\cdot]$, and the same integral technique previously used, we obtain
\begin{align*}
I_3
&\leq 2^r\kappa_r^r\|a\|_\ast^r
\int_{\R^n} |f(x)|^r
\left( \int_0^{C\|a\|_\ast|f(x)|} \frac{d\Phi(t)}{t^r}\right)\, 
dx \\
&\leq 
\frac12 \int_{\R^n} \Phi(C\|a\|_\ast|f(x)|)\,dx,
\end{align*}
and
\begin{align*}
I_4
&\leq 
2^p\kappa_p^p\|a\|_\ast^p
\int_{\R^n} |f(x)|^p
\left( \int_{C\|a\|_\ast|f(x)|}^\infty
\frac{d\Phi(t)}{t^p} \right)\,  dx \\
&\leq  \frac12 \int_{\R^n} \Phi(C\|a\|_\ast|f(x)|)\,dx.
\end{align*}

Consequently,
$$
\int_{\R^n}\Phi(|\C[a,f](x)|)\,dx
\leq  \int_{\R^n}\Phi(C\|a\|_\ast|f(x)|)\,dx,
$$
which, by the definition of Luxemburg norm, implies
$$
\|\C[a,f]\|_{L^\Phi(\R^n)}
\leq  C\,\|a\|_\ast\,\|f\|_{L^\Phi(\R^n)}.
$$
This completes the proof.
\end{proof}

%-----------------------------------------------------------------------------------
\section{Higher-Order Elliptic Operators with {\it VMO} Coefficients}\label{sec5}

Let $\Omega\subset\R^n$, with $n\geq 3$, be an open domain. The restriction  $n\geq 3$ is imposed to avoid the logarithmic behaviour of the fundamental solution. 

In this section, we study the
following linear elliptic system of order $2b$, where $b\ge1$:
\begin{equation}\label{syst}
\mathcal{L}(x,D)\mathbf{u}
:=
\sum_{|\alpha|=2b}
\mathbf{A}_\alpha(x)\,D^\alpha\mathbf{u}(x)
=
\mathbf{f}(x),
\qquad
x\in\Omega.
\end{equation}

Here, $\bm{u}=(u_1,\ldots,u_m)^\top:\Omega\to\R^m$ is the unknown vector-valued function, with $m\geq 1$, while \mbox{$\bm{f}=(f_1,\ldots,f_m)^\top:\Omega\to\R^m$} is a prescribed
vector-valued function. Moreover,
$$
\bA_\alpha(x)=\{a_{jk}^{(\alpha)}(x)\}_{j,k=1}^m
$$
denotes an $m\times m$ matrix of measurable functions on $\Omega$, where
$$
a_{jk}^{(\alpha)}=
a_{jk}^{\alpha_1\cdots\alpha_n}.
$$

We introduce the component-wise differential operators
\begin{equation}\label{eq-hom}
l_{jk}(x,D) =
\sum_{|\alpha|=2b}
a_{jk}^{(\alpha)}(x)\,D^\alpha,
\qquad
j,k=1,\ldots,m.
\end{equation}
With this notation, the system \eqref{syst} can be rewritten in component form as
\begin{equation}\label{syst1}
\sum_{k=1}^m
l_{jk}(x,D)\,u_k(x)
=
f_j(x),
\qquad
j=1,\ldots,m.
\end{equation}
We assume that the coefficients
$a_{jk}^{(\alpha)}\in VMO\cap L^\infty(\Omega)$, and denote by
$$
\eta_{\bA}(R)
:=
\sum_{j,k=1}^m
\sum_{|\alpha|=2b}
\eta_{a_{jk}^{(\alpha)}}(R), %\quad \lim_{R\to0}\eta_{\bA (R)=0.
$$
their joint $VMO$-modulus; it is clear that 
\[\quad \lim_{R\to0}\eta_{\bA}(R)=0.
\]
Moreover, we define the uniform bound
$$
\|\bA\|_{\infty,\Omega}
:=
\max_{j,k=1,\ldots,m}
\max_{|\alpha|=2b}
\|a_{jk}^{(\alpha)}\|_{\infty,\Omega}.
$$

We say that $\bm{u}\in W^{2b}_\Phi(\Omega;\R^m)$ is a \emph{local strong solution} of \eqref{syst} if it satisfies \eqref{syst1} almost everywhere in every subdomain $\Omega'\Subset\Omega$. Recall that the Orlicz--Sobolev norm is defined by
\begin{equation}\label{SobOrl}
\|\bm{u}\|_{W^{2b}_\Phi(\Omega)}
:= \sum_{|\alpha|\le2b} \|D^\alpha\bm{u}\|_{L^\Phi(\Omega)}.
\end{equation}
  
We now introduce  the principal symbol of the operator: for each $x\in\Omega$ and $\zeta\in\R^n$, let
$$
l_{jk}(x,\zeta) := \sum_{|\alpha|=2b} a_{jk}^{(\alpha)}(x)\,\zeta^\alpha,
$$
where $\zeta^\alpha=\zeta_1^{\alpha_1}\cdots\zeta_n^{\alpha_n}$.
We assume that the system \eqref{syst} is
\emph{uniformly elliptic} in the sense of 
Agmon--Douglis--Nirenberg
(cf. \cite{ChFF,DgN,PSf2,Sf1}), namely, there exists a constant
$\delta>0$ such that
\begin{equation}\label{hom+ellip}
\det\{l_{jk}(x,\zeta)\}_{j,k=1}^m
\geq 
\delta\,|\zeta|^{2bm}
\qquad
\text{for a.e. } x\in\Omega
\text{ and every  } \zeta\neq 0.
\end{equation}

We are now in a position to state the main regularity result of this section.

\begin{thm}
Suppose that the uniform ellipticity condition \eqref{hom+ellip} holds and
$a_{jk}^{(\alpha)}\in VMO\cap L^\infty(\Omega)$ for every $|\alpha|=2b$ and $j,k=1,\ldots,m$.

Assume further that
$$
\bm{f}\in L^\Phi(\Omega;\R^m),
\qquad
\Phi\in\Delta_2\cap\nabla_2,
$$
and that
$\bm{u}\in W^{2b}_\Phi(\Omega;\R^m)$ is a strong solution of
\eqref{syst}. Then, for every pair of subdomains
$\Omega'\Subset\Omega''\Subset\Omega$, there exists a constant
$C>0$, depending only on $n$, $m$, $b$, $\delta$, $\|\bA\|_{\infty,\Omega}$, the $VMO$  modulus $\eta_{\bA}$, and $\dist(\Omega',\partial \Omega'')$ such that
\begin{equation}\label{estim1}
\|\bm{u}\|_{W^{2b}_\Phi(\Omega')}
\leq 
C\left(
\|\bm{f}\|_{L^\Phi(\Omega'')} + \|\bm{u}\|_{L^\Phi(\Omega'')} \right).
\end{equation}
\end{thm}

\begin{proof}
To prove the theorem, we freeze the coefficients of \eqref{syst} at a point
$x_0\in\Omega$ and consider the constant-coefficients elliptic differential
operator of order $2bm$ defined by
\begin{equation}
L(x_0,D) :=\det\mathcal{L}(x_0,D) =
\det\left\{\sum_{|\alpha|=2b}
a_{jk}^{(\alpha)}(x_0)\,D^\alpha\right\}.
\end{equation}

Since $L(x_0,D)$ is a constant-coefficients elliptic operator, it admits a fundamental solution by the Ehren\-preis--Malgrange theorem and we denote it by $\widetilde{\Gamma}(x_0;\cdot)$. Its explicit form depends on the parity of the dimension $n$ (see, for instance, \cite{DgN,J,PSf2}). More precisely, if $n$ is odd, then
$$
\widetilde{\Gamma}(x_0;x-y)
=
|x-y|^{2bm-n}
P\!\left(
x_0;
\frac{x-y}{|x-y|}
\right),
$$
where $P(x_0;\cdot)$ is a real-analytic function on the unit sphere
$\SF^{n-1}$; if $n$ is even, then one introduces an auxiliary variable
$x_{n+1}$ and regards functions
$f(x,x_{n+1})$ as being independent of $x_{n+1}$ for each fixed
$x\in\R^n$. 
This explains the assumption $n\geq 3,$ which avoids the logarithmic term appearing in even dimensions.

Let $\{L_{jk}(x_0,\zeta)\}_{j,k=1}^m$ denote the cofactor matrix of
$\{l_{jk}(x_0,\zeta)\}_{j,k=1}^m$. Observe that, for each fixed
$j,k\in\{1,\ldots,m\}$, the operator $L_{jk}(x_0,D)$ is either a
differential operator of order $2b(m-1)$ or the zero operator.
Since 
$$
\sum_{k=1}^m
l_{ik}(x_0,\zeta)\,
L_{jk}(x_0,\zeta) =
\delta_{ij}\,L(x_0,\zeta),
$$
it follows  that the fundamental matrix $\Gamma(x_0;x)$ associated with the
constant-coefficients operator $\mathcal{L}(x_0,D)$
%of the system \eqref{syst} 
has entries
$$
\Gamma_{jk}(x_0;x) =
L_{jk}(x_0,D)\,\widetilde{\Gamma}(x_0;x).
$$
%Here, for each fixed $j,k\in\{1,\ldots,m\}$, the operator $L_{jk}(x_0,D)$ is either a differential operator of order $2b(m-1)$ or the zero operator. 
We note that some cofactors
$L_{jk}(x_0,\zeta)$ may vanish identically as polynomials in $\zeta$.
Since $\widetilde\Gamma(x_0;\cdot)$ is homogeneous of degree
$2bm-n$, it follows that $ 
D^\beta\Gamma_{jk}(x_0;\cdot)$ 
is homogeneous of degree
$$
2bm-n-2b(m-1)-|\beta| =2b-n-|\beta|.
$$

\begin{rem}\label{kern}
For $|\beta|=2b$ the kernels
$D^\beta\Gamma_{jk}(x_0;\cdot)$ are homogeneous of degree $-n$ and satisfy the assumptions of Definition~\ref{dVCZ}.
\end{rem}

Let us fix a ball $\B_r\Subset\Omega$. It is sufficient to prove the estimate
for $\bm{v}\in C_0^\infty(\B_r;\R^m)$ and then extend it to arbitrary
strong solutions by a standard density argument.
Indeed, since $\Phi$ satisfies the $\Delta_2$-condition, the
Orlicz--Sobolev space ${W}^{2b}_\Phi(\B_r;\R^m)$ is separable, and
$C_0^\infty(\B_r;\R^m)$ is dense in it. Therefore, for every
$\bm{v}\in W^{2b}_\Phi(\B_r;\R^m),$  there exists a sequence
$\{\bm{v}_k\}\subset C_0^\infty(\B_r;\R^m)$ such that  
$ \bm{v}_k\to\bm{v}$ in $W^{2b}_\Phi(\B_r;\R^m)$ 
(see for instance \cite[Theorem 8.21]{AF}). 
Applying the estimate proved below to $\bm{v}_k$ and passing to the limit by the continuity of the differential operators and of the Luxemburg norm, we obtain the desired estimate for $\bm{v}$ (see also \cite{CGMP,KZ1,KZ2,VK,RR2}).

Applying the constant-coefficients  operator $\mathcal{L}(x_0,D)$ to $\bm{v}$, we obtain
\begin{align*}
\mathcal{L}(x_0,D)\bm{v}(x)
&=\bigl(\mathcal{L}(x_0,D)-\mathcal{L}(x,D)\bigr)\bm{v}(x)
+\mathcal{L}(x,D)\bm{v}(x)\\
&=\sum_{|\alpha|=2b}
\bigl(\bA_\alpha(x_0)-\bA_\alpha(x)\bigr)
D^\alpha\bm{v}(x)
+\mathcal{L}(x,D)\bm{v}(x).
\end{align*}

By the \textit{interior representation formula} for constant-coefficients elliptic systems associated with the fundamental matrix $\Gamma(x_0;\cdot)$,
(see  \cite{LU,PSf2}), we have
\begin{equation}\label{eq-v}
\bm{v}(x)
= \int_{\B_r}
\Gamma(x_0;x-y)\,
\mathcal{L}(x_0,D)\bm{v}(y)\,dy. 
\end{equation}
%where $\Gamma(x_0;\cdot)$ denotes the fundamental matrix associated with the constant-coefficient operator obtained by freezing $\mathcal{L}(x,D)$ at $x_0\in\Omega$.

Define
$$
g(y)=
\begin{cases}
\displaystyle
\sum_{|\alpha'|=2b}
\bigl(\bA_{\alpha'}(x_0)-\bA_{\alpha'}(y)\bigr)
D^{\alpha'}\bm{v}(y)
+\mathcal{L}(y,D)\bm{v}(y),
& y\in\B_r,\\[6pt] 
0, & y\notin\B_r.
\end{cases}
$$

Hence, \eqref{eq-v} can be rewritten as
$$
\bm{v}(x)
= \bigl(\Gamma(x_0;\cdot)*g\bigr)(x).
$$

Differentiating the representation formula and using standard Fourier transform arguments, together with integration by parts (see, for instance, \cite{ChFrL1}), we obtain
$$
D^\alpha\bm{v}(x)=
\operatorname{p.v.} \,  \bigl(D^\alpha\Gamma(x_0;\cdot)*g\bigr)(x)+
g(x)\sum_{s=1}^n
\int_{\SF^{n-1}}
D^{\gamma_s}\Gamma(x_0;y)\,\nu_s\,d\sigma_y,
$$
where $ 
\gamma_s=
(\alpha_1,\ldots,\alpha_{s-1},
\alpha_s-1,\alpha_{s+1},\ldots,\alpha_n)$ 
is a multi-index of order $|\gamma_s|=2b-1$, and $\nu_s$ denotes the
$s$-th component of the outward unit normal to $\SF^{n-1}$.

Unfreezing the coefficients by setting $x_0=x$, we obtain
\begin{align*}
D^\alpha\bm{v}(x)
&=
\operatorname{p.v.}
\int_{\B_r}
D^\alpha\Gamma(x;x-y)
\Bigg[
\sum_{|\alpha'|=2b}
\bigl(\bA_{\alpha'}(x)-\bA_{\alpha'}(y)\bigr)
D^{\alpha'}\bm{v}(y)
\\
&\qquad\qquad\qquad\qquad
+\mathcal{L}(y,D)\bm{v}(y)
\Bigg]\,dy
\\
&\qquad
+\mathcal{L}(x,D)\bm{v}(x)
\sum_{s=1}^n
\int_{\SF^{n-1}}
D^{\gamma_s}\Gamma(x;y)\,\nu_s\,d\sigma_y.
\end{align*}

The entries of the matrix-valued kernel
$D^\alpha\Gamma(x;\cdot)$, with $|\alpha|=2b$, are variable
Calder\'on--Zygmund kernels, as noted  in Remark\stilde\ref{kern}
%in the sense of Definition~\ref{dVCZ}
(cf.~\cite{PSf2}). Consequently, they define the singular integral operators
\begin{align*}
\K_\alpha(\mathcal{L}(x,D)\bm{v})(x)
&=
\operatorname{p.v.}
\int_{\B_r}
D^\alpha\Gamma(x;x-y)\,
\mathcal{L}(y,D)\bm{v}(y)\,dy,\\
\C_\alpha[\bA_{\alpha'},D^{\alpha'}\bm{v}](x)
&=
\operatorname{p.v.}
\int_{\B_r}
D^\alpha\Gamma(x;x-y)\,
\bigl(\bA_{\alpha'}(x)-\bA_{\alpha'}(y)\bigr)
D^{\alpha'}\bm{v}(y)\,dy.
\end{align*}
Therefore,
\begin{equation*}
\begin{split}
D^\alpha\bm{v}(x)
&=
\sum_{|\alpha'|=2b}
\C_\alpha[\bA_{\alpha'},D^{\alpha'}\bm{v}](x)
+\K_\alpha(\mathcal{L}(x,D)\bm{v})(x)
%\\ &\qquad
+\mathcal{L}(x,D)\bm{v}(x)
\sum_{s=1}^n
\int_{\SF^{n-1}}
D^{\gamma_s}\Gamma(x;y)\,\nu_s\,d\sigma_y.
\end{split}
\end{equation*}

Applying the boundedness results \eqref{eq-Kfphi2} and
\eqref{eq-CAFphi2} to the above representation %\textcolor{red}{write something about why there is not the estimate of last integral?} 
we obtain
$$
\|D^{2b}\bm{u}\|_{L^\Phi(\B_r)}
\leq 
C\left(
\sum_{|\alpha'|=2b}
\|\bA_{\alpha'}\|_{*,\B_r}\,
\|D^{\alpha'}\bm{u}\|_{L^\Phi(\B_r)} +
\|\mathcal{L}(\cdot,D)\bm{u}\|_{L^\Phi(\B_r)}
\right).
$$

The $VMO$ assumption on the higher-order coefficients
(cf.~\cite{PSf1}) implies that, 
%\textcolor{red}{write something about the equivalence between local estimate \eqref{local} and the corresponding Orlicz one} 
for every $\varepsilon>0$, there exists
$r_0=r_0(\varepsilon,\eta_{\bA})>0$ such that, whenever $r<r_0$,
$$
\|D^{2b}\bm{u}\|_{L^\Phi(\B_r)}
\leq
C\left( \varepsilon
\sum_{|\alpha'|=2b}
\|D^{\alpha'}\bm{u}\|_{L^\Phi(\B_r)} +
\|\mathcal{L}(\cdot,D)\bm{u}\|_{L^\Phi(\B_r)}\right).
$$

Since all the norms on the finite-dimensional space of derivatives of order $2b$ are equivalent, and 
$$
\sum_{|\alpha'|=2b}
\|D^{\alpha'}\bm{u}\|_{L^\Phi(\B_r)}
\leq C\,\|D^{2b}\bm{u}\|_{L^\Phi(\B_r)},
$$
it follows that
$$
\|D^{2b}\bm{u}\|_{L^\Phi(\B_r)}
\leq 
C\left( \varepsilon\,
\|D^{2b}\bm{u}\|_{L^\Phi(\B_r)} +
\|\mathcal{L}(\cdot,D)\bm{u}\|_{L^\Phi(\B_r)}
\right).
$$

Choosing $\varepsilon>0$ sufficiently small, the first term on the right-hand
side can be absorbed into the left-hand side, yielding
\begin{equation}\label{final}
\|D^{2b}\bm{u}\|_{L^\Phi(\B_r)}
\leq 
C\,\|\mathcal{L}(\cdot,D)\bm{u}\|_{L^\Phi(\B_r)}.
\end{equation}

Let $\theta\in(0,1)$ and define
$ \theta':=\theta(3-\theta)/2>\theta, $
%\textcolor{blue}{Fix a ball $\B_r\Subset\Omega$ is it a repetition?}, 
and let
$\varphi\in C_0^\infty(\B_r)$ be a cut-off function defined by
$$
\varphi(x)=
\begin{cases}
1, & x\in\B_{\theta r},\\
0, & x\notin\B_{\theta' r}.
\end{cases}
$$
Moreover,
$$
|D^s\varphi(x)| \leq
C(s)\,[\theta(1-\theta)r]^{-s},
\qquad
1\leq s\leq 2b,
$$
since $ 
\theta'-\theta=\frac{\theta(1-\theta)}{2}.$ 

For $\bm{u}\in W^{2b}_{\Phi, \loc}(\Omega;\R^m)$, it follows that $\bm{v}:=\varphi\bm{u}\in W^{2b}_\Phi(\B_r;\R^m)$ with compact support in $\B_r$. 

Therefore, we may apply
\eqref{final} to $\bm{v}$ obtaining 
\begin{equation}\label{phiu}
\|D^{2b}\bm{u}\|_{L^\Phi(\B_{\theta r})}
\leq 
\|D^{2b}\bm{v}\|_{L^\Phi(\B_{\theta' r})}
\leq 
C\,\|\mathcal{L}(\cdot,D)\bm{v}\|_{L^\Phi(\B_{\theta' r})}
\leq 
C\,\|\mathcal{L}(\cdot,D)(\varphi\bm{u})\|_{L^\Phi(\B_{\theta' r})}.
\end{equation}

Applying Leibniz's rule  (cf. \cite{PSf2}), we obtain
\allowdisplaybreaks
\begin{align*}
\mathcal{L}(x,D)(\varphi\bm{u})(x)
&=\varphi(x)
\sum_{|\alpha|=2b}
\bA_\alpha(x)\,D^\alpha\bm{u}(x)+
\bm{u}(x)\sum_{|\alpha|=2b}
\bA_\alpha(x)\,D^\alpha\varphi(x) \\
&+\sum_{|\alpha|=2b}
\sum_{0<|\beta|<2b}
c_{\alpha,\beta}\,
\bA_\alpha(x)\, D^\beta\bm{u}(x)\,
D^{\alpha-\beta}\varphi(x),
\end{align*}
where $c_{\alpha,\beta}$ are the multinomial coefficients arising from Leibniz's formula.

Therefore, by \eqref{phiu} we obtain
\begin{align*}
\|D^{2b}\bm{u}\|_{L^\Phi(\B_{\theta r})}
\leq C\,\Bigg\{
\|\bm{f}\|_{L^\Phi(\B_{\theta'r})}+
\sum_{s=1}^{2b-1}
\frac{\|D^s\bm{u}\|_{L^\Phi(\B_{\theta'r})}}
{[\theta(1-\theta)r]^{\,2b-s}} +
\frac{\|\bm{u}\|_{L^\Phi(\B_{\theta'r})}}
{[\theta(1-\theta)r]^{\,2b}}\Bigg\}.
\end{align*}

%%%%%%%%%%%%%%%%%%%

Multiplying both sides by $[\theta(1-\theta)r]^{2b}$, using the inequality
$
\theta(1-\theta)\leq 2\theta'(1-\theta'),
$
and absorbing the resulting powers of $2$ into the constant $C$, we obtain
\begin{align*}
[\theta(1-\theta)r]^{2b}
\|D^{2b}\bm{u}\|_{L^\Phi(\B_{\theta r})}
&\leq C\,\Bigg\{
[\theta'(1-\theta')r]^{2b}
\|\bm{f}\|_{L^\Phi(\B_{\theta'r})}\\
&\qquad
+\sum_{s=1}^{2b-1}
[\theta'(1-\theta')r]^s
\|D^s\bm{u}\|_{L^\Phi(\B_{\theta'r})}
+\|\bm{u}\|_{L^\Phi(\B_{\theta'r})}
\Bigg\}.
\end{align*}

Define the seminorms
$$
\Theta_s :=
\sup_{\theta\in(0,1)}
[\theta(1-\theta)r]^s
\|D^s\bm{u}\|_{L^\Phi(\B_{\theta r})},
\qquad
0\leq s\leq 2b.
$$
Then, the previous inequality can be rewritten as
\begin{equation}\label{02beta}
\Theta_{2b}
\leq 
C\,\left\{
r^{2b}\|\bm{f}\|_{L^\Phi(\B_r)}
+\Theta_0 +\sum_{s=1}^{2b-1}\Theta_s
\right\}.
\end{equation}

Applying  the interpolation inequality in Orlicz spaces
(cf.~\cite{AK1}), valid on balls, we obtain that, for every
$1\leq  s\le2b-1$ and every $\mu>0$, there exists a constant
$C(s)>0$, independent of $\mu$, $r$, and $\theta$, such that
\begin{equation}\label{interp}
\|D^s\bm{u}\|_{L^\Phi(\B_{\theta r})}
\leq 
\mu\,\|D^{2b}\bm{u}\|_{L^\Phi(\B_{\theta r})} +
\frac{C(s)}{\mu^{\frac{s}{2b-s}}}
\|\bm{u}\|_{L^\Phi(\B_{\theta r})}.
\end{equation}
Multiplying both sides of \eqref{interp} by
$[\theta(1-\theta)r]^s$ yields
\begin{align*}
[\theta(1-\theta)r]^s
\|D^s\bm{u}\|_{L^\Phi(\B_{\theta r})}
&\leq \mu\,
[\theta(1-\theta)r]^{-(2b-s)}
[\theta(1-\theta)r]^{2b}
\|D^{2b}\bm{u}\|_{L^\Phi(\B_{\theta r})} \\
&+\frac{C(s)}{\mu^{\frac{s}{2b-s}}}
[\theta(1-\theta)r]^s
\|\bm{u}\|_{L^\Phi(\B_{\theta r})}.
\end{align*}

Choosing 
$$
\mu:=\varepsilon[\theta(1-\theta)r]^{2b-s}, 
$$
where $\varepsilon>0$ is arbitrary, we obtain
\begin{align*}
[\theta(1-\theta)r]^s
\|D^s\bm{u}\|_{L^\Phi(\B_{\theta r})}
&\leq\varepsilon
[\theta(1-\theta)r]^{2b}
\|D^{2b}\bm{u}\|_{L^\Phi(\B_{\theta r})}
+\frac{C(s)}{\varepsilon^{\frac{s}{2b-s}}}
\|\bm{u}\|_{L^\Phi(\B_{\theta r})}.
\end{align*}
Taking the supremum over $\theta\in(0,1)$, we conclude that
$$
\Theta_s\leq
\varepsilon\,\Theta_{2b}+
\frac{C(s)}{\varepsilon^{\frac{s}{2b-s}}}\Theta_0,
\qquad
1\leq  s\le2b-1.
$$

Recalling the definitions of $\Theta_{2b}$ and $\Theta_0$ and choosing
$\theta=\frac12$, we obtain the following Caccioppoli-type estimate:
\begin{equation}\label{estim2}
\|D^{2b}\bm{u}\|_{L^\Phi(\B_{r/2})}
\leq C\,\left(
\|\bm{f}\|_{L^\Phi(\B_r)}
+ r^{-2b}\|\bm{u}\|_{L^\Phi(\B_r)}
\right).
\end{equation}

Finally, estimate \eqref{estim1} follows from \eqref{estim2} by covering
$\Omega'$ with finitely many balls of radius $\B_{r/2}\Subset \Omega''$  where $
r<\dist (\Omega',\partial\Omega'').$
This completes the proof.
\end{proof}

\paragraph{\textbf{\small Acknowledgements.}} The research of  A.~Gogatishvili was partially supported by the grant project 23-04720S of the Czech Science Foundation (GA\v{C}R); the Institute of Mathematics, CAS, is supported by RVO:67985840; and by the Shota Rustaveli National Science Foundation (SRNSF), grant no: FR22-17770.

P.~Salerno and L.~Softova are members of INDAM--GNAMPA.
Part of the research presented in this paper was carried out during the authors'
visit to the \emph{Institute of Mathematics of the Czech Academy of Sciences}.
The authors gratefully acknowledge the Institute and its staff for their
hospitality and support.

\paragraph{\textbf{Compliance with Ethical Standards}.}
The authors declare that they have no conflict of interest.

\paragraph{\textbf{Data Availability Statement}.}
Data sharing is not applicable to this article as no datasets were generated or analysed during the current study.

\paragraph{\textbf{Author Contributions.}} All authors contributed equally to this work. All authors participated in the conception of the study, the development of the mathematical results, the writing and revision of the manuscript, and approved the final version for publication.

\bibliographystyle{amsplain}

\end{document}